\documentclass[11pt]{article}
\usepackage{mathrsfs}
\usepackage{latexsym,amssymb,amsmath,amsfonts,amsthm}
\usepackage{graphicx}
\newcommand{\R}{{\mathbb R}}

\newcommand{\ds}{\displaystyle}
\newcommand{\no}{\nonumber}
\newcommand{\be}{\begin{eqnarray}}
\newcommand{\ben}{\begin{eqnarray*}}
\newcommand{\en}{\end{eqnarray}}
\newcommand{\enn}{\end{eqnarray*}}
\newcommand{\ba}{\backslash}
\newcommand{\pa}{\partial}

\newcommand{\ov}{\overline}

\newcommand{\G}{\Gamma}

\newcommand{\Om}{\Omega}
\newcommand{\om}{\omega}

\newcommand{\al}{\alpha}
\newcommand{\la}{\lambda}
\newcommand{\La}{\Lambda}
\newcommand{\wi}{\widetilde}

\newtheorem{theorem}{Theorem}[section]
\newtheorem{lemma}[theorem]{Lemma}

\newtheorem{remark}[theorem]{Remark}

\begin{document}
\renewcommand{\theequation}{\arabic{section}.\arabic{equation}}
\begin{titlepage}
\title{\bf Inverse scattering by an inhomogeneous penetrable obstacle
in a piecewise homogeneous medium}
\author{Xiaodong Liu,\ \ Bo Zhang\\
LSEC and Institute of Applied Mathematics, AMSS\\
Chinese Academy of Sciences,\\
Beijing 100190, China\\
{\sf lxd230@163.com} (XL),\ {\sf b.zhang@amt.ac.cn} (BZ)}
\date{}
\end{titlepage}
\maketitle

\begin{abstract}
This paper is concerned with the inverse problem of scattering of
time-harmonic acoustic waves by an inhomogeneous penetrable obstacle
in a piecewise homogeneous medium.
The well-posedness of the direct problem is first established by using
the integral equation method.
We then proceed to establish two tools that play an important role for
the inverse problem: one is a mixed reciprocity relation and the other
is a priori estimates of the solution on some part of the interfaces
between the layered media.
For the inverse problem, we prove in this paper that both the penetrable
interfaces and the possible inside inhomogeneity can be uniquely determined
from a knowledge of the far field pattern for incident plane waves.

\vspace{.2in} {\bf Keywords:} Uniqueness, piecewise homogeneous medium,
penetrable obstacle, unique continuation principle, Holmgren's uniqueness
theorem, inverse scattering.
\end{abstract}

\section{Introduction}
\setcounter{equation}{0}

In this paper, we consider the scattering of time-harmonic acoustic
plane waves by an inhomogeneous, penetrable obstacle in a piecewise homogeneous
surrounding medium. In many practical applications, the background medium
might not be homogeneous and then may be modeled as a layered medium.
We might think of a problem in medical imaging where we have
a damaged tissue (which can be modeled as a penetrable obstacle)
inside the human body. It is clear that the human body is not a
homogeneous structure and may be modeled as a piecewise homogeneous medium.
Therefore, one possible model would then be a penetrable obstacle buried in a
piecewise homogeneous medium.

For simplicity, and without loss of generality, in this paper we
restrict ourself to the case where the obstacle is buried in a two-layered piecewise
homogeneous medium, as shown in Figure \ref{fig1}.
Precisely, let $\Om_{2}\subset\R^3$ be an open bounded region with a $C^{2}$ boundary $S_{1}$
such that the background $\R^3\ba\ov{\Om_{2}}$ is divided by means of a closed
$C^{2}$ surface $S_{0}$ into two connected domains $\Om_{0}$ and
$\Om_{1}$. Here, $\Om_{0}$ is a unbounded homogeneous medium,
$\Om_{1}$ is a bounded homogeneous medium and $\Om_2$ is a penetrable obstacle.
Let $\Om$ denote the complement of $\Om_{0}$, that is, $\Om:=\R^{3}\ba\ov{\Om_0}$.
Choose a large ball $B_{R}$ centered at the origin such that $\ov{\Om}\subset B_{R}$
and let $\Om_{R}:=B_{R}\ba\ov{\Om}$.

The problem of scattering of time-harmonic acoustic waves in a two-layered
background medium in $\R^3$ can be modeled by
 \be
 \label{0HE0}\Delta u+k_{0}^{2}u=0&&\quad \mbox{in}\;\;\Om_{0},\\
 \label{0HE1}\Delta v+k_{1}^{2}v=0&&\quad \mbox{in}\;\;\Om_{1},\\
 \label{0HE2}\Delta w+k_{2}^{2}nw=0&&\quad\mbox{in}\;\;\Om_{2},\\
 \label{0tbc}u-v=0,\;\;\frac{\pa u}{\pa\nu}-\la_0\frac{\pa v}{\pa\nu}=0&&\quad\mbox{on}\;\;S_0,\\
 \label{0tbc1}v-w=0,\;\;\frac{\pa v}{\pa\nu}-\la_1\frac{\pa w}{\pa\nu}=0&&\quad \mbox{on}\;\;S_{1},\\
 \label{0rc}\lim_{r\rightarrow\infty}r\left(\frac{\pa u^s}{\pa r}-ik_0u^s\right)=0,&&\quad r=|x|,
 \en
where $\nu$ is the unit outward normal to the interface $S_{j}\;(j=0,1)$ and
$n\in C^{0,\al}(\ov{\Om_{2}}), 0<\al<1$, is the refractive index of an inhomogeneous
medium with $\Re(n)>0$ and $\Im(n)\geq0$.
Here, the total field $u=u^{s}+u^{i}$ is given as the
sum of the unknown scattered wave $u^{s}$ which is required to
satisfy the Sommerfeld radiation condition (\ref{0rc}) and the incident
plane wave $u^{i}=e^{ik_{0}x\cdot d}$,
$k_{j}$ is the positive wave
number given by $k_{j}=\om_{j}/c_{j}$ in terms of the frequency
$\om_{j}$ and the sound speed $c_{j}$ in the corresponding medium $\Om_{j}\;(j=0,1,2)$.
The distinct wave numbers $k_{j}\;(j=0,1,2)$ correspond to the fact that the medium
consists of several physically different materials.
On the interfaces $S_{0}$ and $S_1$, the so-called "transmission
conditions" (\ref{0tbc}) and (\ref{0tbc1}) with two positive constants
$\la_0$ and $\la_1$ are imposed, respectively, which represent the continuity of the medium and
equilibrium of the forces acting on them.
\begin{figure}[htbp]
\begin{center}
\scalebox{0.35}{\includegraphics{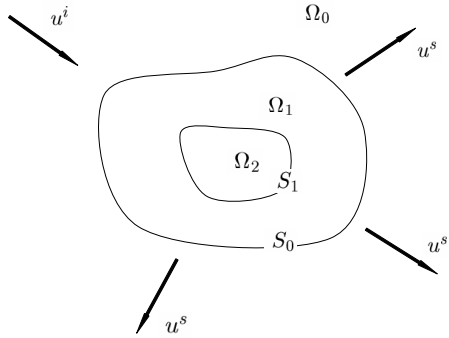}} \caption{Scattering in a
two-layered background medium }\label{fig1}
\end{center}
\end{figure}

The {\em direct problem} is to look for a set of functions $(u, v, w)$ satisfying
(\ref{0HE0})-(\ref{0rc}). We will establish the well-posedness of the direct problem,
employing the integral equation method in the next section.
We will also establish a mixed reciprocity relation and prove a priori estimates of
the solution on some part of the interfaces $S_{j}$ ($j=0,1$). These results will
play an important role in the proof of the uniqueness results in the inverse problem.

It is well known that $u^{s}(x)$ has the following asymptotic representation
 \be\label{s-f}
u^{s}(x,d)=\frac{e^{ik_{0}|x|}}{|x|}\left\{u^{\infty}(\widehat{x},d)
+O(\frac{1}{|x|})\right\}\;\qquad \mbox{as }\;|x|\rightarrow\infty
 \en
uniformly for all directions $\widehat{x}:=x/|x|$, where the
function $u^{\infty}(\widehat{x},d)$ defined on the unit sphere $S$
is known as the far field pattern with $\widehat{x}$ and $d$
denoting, respectively, the observation direction and the incident direction.

The {\em inverse problem} we consider in this paper is, given the
wave numbers $k_{j}$ ($j=0,1,2$), the positive constants $\la_{0},
\la_1$ and the far field pattern $u^{\infty}(\widehat{x},d)$ for all
incident plane waves with incident direction $d\in S$, to determine
the interfaces $S_{j}$ ($j=0,1$) and the refractive index $n$.
Precisely, we will study the uniqueness issue and prove that
the interfaces $S_{j}$ ($j=0,1$) and the inhomogeneity $n$ can be uniquely identified
from a knowledge of the far field pattern.
It should be remarked that the uniqueness issue for inverse problems is of theoretical
interest and is required in order to proceed to efficient numerical methods of solutions.

For the unique determination of an inhomogeneity with compact support in a
homogeneous background medium, we refer the reader to H\"{a}hner \cite{Ha96},
Nachman \cite{N88}, Novikov \cite{No88} and Ramm \cite{Ramm88,Ramm92}.
We also refer the reader to the monographs of Colton and Kress \cite{CK} and
Kirsch \cite{KBook} and the habilitationsschrift of H\"{a}hner \cite{Hbook}
for a comprehensive discussion. Kirsch and P\"{a}iv\"{a}rinta \cite{KP} gave the first
uniqueness result for determining the interior inhomogeneity in the case of
a known inhomogeneous background medium.

However, to the authors' knowledge, there are few uniqueness results for the inverse
obstacle scattering in a piecewise homogeneous medium. In the case when the obstacle
$\Om_{2}$ is impenetrable, based on a generalization of the mixed reciprocity relation,
Liu, Zhang and Hu \cite{LZH} proved that if the interface $S_{0}$ is known then
the obstacle and its physical property can be uniquely determined by the far field pattern.
Later, motivated by Kirsch and Kress \cite{KK93} and Kirsch and P\"{a}iv\"{a}rinta \cite{KP}
for the uniqueness proof of determining a penetrable boundary,
Liu and Zhang \cite{LZ2} extended this result to the case when the interface $S_{0}$
is unknown and proved that the interface $S_{0}$ can also be uniquely recovered.
Note that all the results in \cite{LZH, LZ2} are also available for the case of a multilayered
medium and can be proved similarly. Thus, Liu and Zhang \cite{LZ2} have in fact proved
a uniqueness result for the case when the obstacle $\Om_{2}$ is penetrable with
a homogeneous interior. This result has also been proved by Athanasiadis, Ramm and
Stratis \cite{ARS} and Yan \cite{YG} using a different method.

In this paper, we consider the case where the obstacle $\Om_2$ is penetrable with
an inhomogeneous interior and use the ideas in \cite{KK93, KP, LZ2} to prove
the unique determination of the interfaces $S_{j}$ ($j=0,1$) in Section \ref{sec3}.
In Section \ref{sec4}, we will show that the refractive index $n$ is also uniquely determined.
To do this, we first establish a completeness result, that is, the normal derivatives
$\{{\pa w(\cdot;d)}/{\pa\nu},\;d\in S\}$, corresponding to incident plane waves
with all directions $d\in S$, are complete in $L^{2}(S_1)$. Based on this result, we then
establish an orthogonality relation that enables us to prove the unique determination
of the inhomogeneity $n$ with help of the ideas from the fundamental work of
Sylvester and Uhlmann \cite{SU87}.


\section{The direct scattering problem}\label{sec2}
\setcounter{equation}{0}

In this section we will establish the well-posedness of the direct problem via
the integral equation method. We also prove a mixed reciprocity relation
and a priori estimates on some part of the interfaces $S_0,S_1$ of the solution
with the help of its explicit representation in a combination of layer and volume potentials.
We assume hereafter that $k_{0}$, $k_{1}$, $k_{2}$, $\la_{0}$, $\la_{1}$ are given positive
numbers and that $k_{2}^{2}$ is not a Neumann eigenvalue of $\Delta w+k_{2}^{2}nw=0$ in $\Om_2$.
In this paper, we shall use $C$ to denote generic constants whose values may change in different
inequalities but always bounded away from infinity.

\begin{remark}\label{r1} {\rm
The assumption on $k_{2}^{2}$ holds provided the refractive index $n$
satisfies one of the following two conditions:
\begin{description}
\item (1) $\Im(n)>0$ on some non-empty open subset $\Om^{\ast}$ of $\Om_{2}$;
\item (2) $1-n$ is compactly supported in $\Om_2$.
\end{description}
Both cases can be proved by the unique continuation principle.
}
\end{remark}

As incident fields $u^{i}$, plane wave $\ds e^{ik_{0}x\cdot d}$ and point source
$\Phi_{j}(\cdot,z_{j}),\;( z_{j}\in\Om_{j},\;j=0,1)$, (cf. (\ref{Phidef}) below)
are of special interest. Denote by $u^{s}(\cdot,d)$ the scattered field for an
incident plane wave $u^{i}(\cdot,d)$ with incident direction $d\in S$ and by
$u^{\infty}(\cdot,d)$ the corresponding far field pattern. The
scattered field for an incident point source $\Phi_{j}(\cdot,z_{j})$ with
source point $z_{j}\in \Om_{j}$ is denoted by $u^{s}(\cdot, z_{j})$ and the
corresponding far field pattern by $\Phi^{\infty}(\cdot,z_{j})$ $(j=0,1)$.

The direct problem is to look for a set of functions $u\in C^2(\Om_0)\cap C^{1,\al}(\ov{\Om}_0)$,
$v\in C^{2}(\Om_1)\cap C^{1,\al}(\ov{\Om}_1)$ and $w\in C^2(\Om_2)\cap C^{1,\al}(\ov{\Om}_2)$
satisfying the following boundary value problem
 \be
 \label{HE0}\Delta u+k_{0}^{2}u=0&&\qquad \mbox{in}\ \Om_{0},\\
 \label{HE1}\Delta v+k_{1}^{2}v=0&&\qquad \mbox{in}\ \Om_{1},\\
 \label{HE2}\Delta w+k_{2}^{2}nw=0&&\qquad \mbox{in}\ \Om_{2},\\
 \label{tbc0}u-v=f,\;\;\frac{\pa u}{\pa\nu}-\la_0\frac{\pa v}{\pa\nu}=g&&\qquad\mbox{on}\;S_0,\\
 \label{tbc1}v-w=p,\;\;\frac{\pa v}{\pa\nu}-\la_1\frac{\pa w}{\pa\nu}=q&&\qquad \mbox{on}\; S_{1},\\
 \label{rc}\lim_{r\rightarrow\infty}r\left(\frac{\pa u}{\pa r}-ik_{0}u\right)=0&&\qquad r=|x|
 \en
where $f\in C^{1,\al}(S_{0})$, $g\in C^{0,\al}(S_{0})$,
$p\in C^{1,\al}(S_{1})$ and $q\in C^{0,\al}(S_{1})$ are given functions
from H\"{o}lder spaces with exponent $0<\al<1$.

\begin{remark}\label{r2} {\rm
For the scattering problem, if the incident field $u^{i}$ is the plane wave
$\ds e^{ik_{0}x\cdot d}$ or the point source $\Phi_0(\cdot,z_0)$ with $z_0\in\Om_0$
then
$\ds f=-u^{i}|_{S_{0}},\;g=-{\pa u^i}/{\pa\nu}|_{S_0},\;p=0,\;q=0$,
and if the incident field $u^{i}$ is the point source $\Phi_1(\cdot,z_1)$ with $z_1\in\Om_1$
then
$\ds f=u^i|_{S_0},\;g={\pa u^i}/{\pa\nu}|_{S_0},\;p=-u^i|_{S_1},\;q=-{\pa u^i}/{\pa\nu}|_{S_1}$.
}
\end{remark}

\begin{theorem}\label{uni.direct}
The boundary value problem $(\ref{HE0})-(\ref{rc})$ admits at most one solution.
\end{theorem}

\begin{proof}
Clearly, it is enough to show that $u=0$ in $\Om_{0}$, $v=0$ in $\Om_1$ and
$w=0$ in $\Om_2$ for the corresponding homogeneous problem,
that is, $f=g=0$ on $S_0$ and $p=q=0$ on $S_1$.
Applying Green's first theorem over $\Om_{R}$, we obtain that
\ben
 \int_{\pa B_{R}}u\frac{\pa\ov{u}}{\pa\nu}ds
  =\int_{\Om_{R}}\left(u\Delta\ov{u}+|\nabla u|^{2}\right)dx
  +\int_{S_{0}}u\frac{\pa\ov{u}}{\pa\nu}ds.
\enn
Using Green's first theorem over $\Om_{1}$ and $\Om_2$ and taking into
account the transmission conditions (\ref{tbc0}) and (\ref{tbc1}), we have
 \be\label{2.7}
\int_{\pa B_{R}}u\frac{\pa\ov{u}}{\pa\nu}ds
 &=&\int_{\Om_{R}}\left(u\Delta\ov{u}+|\nabla u|^{2}\right)dx
    +\la_{0}\int_{\Om_{1}}\left(v\Delta\ov{v}+|\nabla v|^{2}\right)dx\cr
  &&+\la_{0}\la_{1}\int_{\Om_{2}}\left(w\Delta\ov{w}+|\nabla w|^{2}\right)dx.
 \en
Making use of equations (\ref{HE0})-(\ref{HE2}) and taking the
imaginary part of (\ref{2.7}) we obtain, on noting that
$k_{0}^{2},\;k_{1}^{2},\;k_{2}^{2},\;\la_{0},\;\la_{1}$ are positive numbers
and $\Im{n}\geq0$, that
 \ben
  \Im\int_{\pa B_{R}}u\frac{\pa \ov{u}}{\pa\nu}ds
  =-k_{2}^{2}\la_{0}\la_{1}\Im\int_{\Om_{2}}\ov{n}|w|^{2}dx\geq0.
 \enn
Thus, by Rellich's Lemma \cite{CK}, it follows that $u=0$ in $\R^3\ba B_R$.
By the unique continuation principle, we have $u=0$ in $\Om_{0}$.
Holmgren's uniqueness theorem \cite{Kr} and the homogeneous
transmission boundary conditions (\ref{tbc0}) imply that $v=0$ in $\Om_1$.
Finally, the transmission boundary conditions (\ref{tbc1})
and the assumption on $k_{2}^{2}$ give that $w=0$ in $\Om_2$, which completes the
proof of the theorem.
\end{proof}

For the proof of existence of solutions we need the fundamental solution $\Phi_j$
of the Helmholtz equation with wave number $k_{j}\; (j=0,1,2)$ given by
 \be\label{Phidef}
 \Phi_{j}(x,y)=\frac{e^{ik_{j}|x-y|}}{4\pi |x-y|},\qquad x,y\in\R^{3},\;\; x\neq y.
 \en
For $i=0,1$ and $j=0,1,2$, define the single- and double-layer potentials $S_{i,j}$
and $K_{i,j}$, respectively, by
 \ben
 (\wi{S}_{i,j}\phi)(x):= \int_{S_{i}}\Phi_{j}(x,y)\phi(y)ds(y),\qquad &\ x\in \R^{3}\ba S_{i},\\
 (\wi{K}_{i,j}\phi)(x):= \int_{S_{i}}\frac{\pa\Phi_{j}(x,y)}{\pa\nu(y)}\phi(y)ds(y),
   \qquad &\ x\in \R^{3}\ba S_{i}
 \enn
and the normal derivative operators $K^{'}_{i,j}$ and $T_{i,j}$ by
 \ben
 (\wi{K}^{'}_{i,j}\phi)(x):=\frac{\pa}{\pa\nu(x)}\int_{S_{i}}\Phi_{j}(x,y)\phi(y)ds(y),
   \qquad &\ x\in \R^{3}\ba S_{i},\\
 (\wi{T}_{i,j}\phi)(x)
 :=\frac{\pa}{\pa\nu(x)}\int_{S_{i}}\frac{\pa\Phi_{j}(x,y)}{\pa \nu(y)}\phi(y)ds(y),
   \qquad&\ x\in \R^{3}\ba S_{i}.
 \enn
The restrictions on $S_{i}$ of these operators will be denoted by $S_{i,j}$, $K_{i,j}$,
$K^{'}_{i,j}$ and $T_{i,j}$ ($i=0,1,j=0,1,2$), respectively.
To prove the existence of solutions, we also need the volume potential
 \ben
 (V\phi)(x):=k_{2}^{2}\int_{\Om_{2}}\Phi_{2}(x,y)[n(y)-1]\phi(y)dy,\qquad &\ x\in \R^{3}
 \enn
and its normal derivative operator denoted by $V^\prime$.
For mapping properties of these operators in the classical spaces of continuous and H\"{o}lder
continuous functions we refer the reader to the monographs of Colton and Kress \cite{CK83,CK}.

\begin{theorem}\label{well-posedness}
The boundary value problem $(\ref{HE0})-(\ref{rc})$ has a unique solution. Further, the solution
satisfies the estimate
 \be\label{estimate}
 \|u\|_{1,\al,\ov{\Om}_R}+\|v\|_{1,\al,\ov{\Om}_1}+\|w\|_{1,\al,\ov{\Om}_2}
 \leq C(\|f\|_{1,\al,S_0}+\|g\|_{0,\al,S_0}+\|p\|_{1,\al,S_1}+\|q\|_{0,\al,S_1})
 \en
for some positive constant $C=C(\al)$.
\end{theorem}

\begin{proof}
Following \cite{CK83} and \cite{CK} we seek the unique solution in the form
 \be
 \label{u}u&=&\la_0\wi{K}_{0,0}\psi_{0}+\wi{S}_{0,0}\phi_{0}
          \hspace{3.1cm}\qquad \mbox{in}\;\Om_{0},\\
 \label{v}v&=&\wi{K}_{0,1}\psi_0+\wi{S}_{0,1}\phi_0+\la_1\wi{K}_{1,1}\psi_1+\wi{S}_{1,1}\phi_1
      \qquad \mbox{in}\;\Om_{1},\\
 \label{w}w&=&\wi{K}_{1,2}\psi_{1}+\wi{S}_{1,2}\phi_{1}+Vw
      \hspace{2.45cm}\qquad \mbox{in}\;\Om_{2}
 \en
with four densities $\psi_{0}\in C^{1,\al}(S_{0})$, $\phi_{0}\in C^{0,\al}(S_{0})$,
$\psi_{1}\in C^{1,\al}(S_{1})$, $\phi_{1}\in C^{0,\al}(S_{1})$.
Then from the jump relations we see that the potentials $u,\,v$ and $w$ given by
$(\ref{u})-(\ref{w})$ solve the boundary value problem $(\ref{HE0})-(\ref{rc})$
provided the densities satisfy the system of integral equations:
 \be
 \label{psi0}\psi_{0}+\la(\la_0K_{0,0}-K_{0,1})\psi_{0}+\la(S_{0,0}-S_{0,1})\phi_{0}
    -\la\la_{1}\wi{K}_{1,1}\psi_{1}-\la\wi{S}_{1,1}\phi_{1}=\la f\qquad \mbox{on}\; S_{0},\\
 \no\phi_0+\la\la_0(T_{0,1}-T_{0,0})\psi_0+\la(\la_0K^{'}_{0,1}-K^{'}_{0,0})\phi_0\hspace{6.4cm}\\
 \label{phi0}\hspace{6cm}+\la\la_0\la_{1}\wi{T}_{1,1}\psi_{1}+\la\la_0\wi{K}^{'}_{1,1}\phi_{1}
          =-\la g\qquad \mbox{on}\; S_{0},\\
 \no\psi_{1}+\mu\wi{K}_{0,1}\psi_{0}+\mu\wi{S}_{0,1}\phi_{0}\hspace{10cm}\\
 \label{psi1}\hspace{3.5cm}+\mu(\la_{1}K_{1,1}-K_{1,2})\psi_{1}+\mu(S_{1,1}-S_{1,2})\phi_{1}-\mu Vw
          =\mu p\qquad\mbox{on}\; S_{1},\\
 \no\phi_{1}-\mu\wi{T}_{0,1}\psi_{0}-\mu\wi{K}^{'}_{0,1}\phi_{0}\hspace{9.95cm}\\
 \label{phi1}\hspace{2.5cm}+\mu\la_1(T_{1,2}-T_{1,1})\psi_1+\mu(\la_{1}K^{'}_{1,2}-K^{'}_{1,1})\phi_1
        +\mu\la_{1}V^{'}w=-\mu q\qquad \mbox{on}\;S_{1},
 \en
where $\la={2}/({\la_0+1})$ and $\mu={2}/({\la_1+1})$.

Define the product space
 $$
 X:=C^{1,\al}(\ov{\Om_R})\times C^{1,\al}(\ov{\Om_1})\times C^{1,\al}(\ov{\Om_2})
 \times C^{1,\al}(S_0)\times C^{0,\al}(S_0)\times C^{1,\al}(S_1)\times C^{0,\al}(S_1)
 $$
and introduce the operator $A:X\rightarrow X$ given in the matrix form:
 \ben
\left(\begin{matrix}
 0&0&0&\la_0\wi{K}_{0,0}&\wi{S}_{0,0}&0&0\\
 0&0&0&\wi{K}_{0,1}&\wi{S}_{0,1}&\la_1\wi{K}_{1,1}&\wi{S}_{1,1}\\
 0&0&-V&0&0&\wi{K}_{1,2}&\wi{S}_{1,2}\\
 0&0&0&\la(\la_0K_{0,0}-K_{0,1})&\la(S_{0,0}-S_{0,1})&-\la\la_{1}\wi{K}_{1,1}&-\la\wi{S}_{1,1}\\
 0&0&0&\la\la_0(T_{0,1}-T_{0,0})&\la(\la_{0}K^{'}_{0,1}-K^{'}_{0,0})
      &\la\la_0\la_1\wi{T}_{1,1}&\la\la_0\wi{K}^{'}_{1,1}\\
 0&0&-\mu V&\mu\wi{K}_{0,1}&\mu\wi{S}_{0,1}&\mu(\la_{1}K_{1,1}-K_{1,2})&\mu(S_{1,1}-S_{1,2})\\
 0&0&\mu\la_{1}V^{'}&-\mu\wi{T}_{0,1}&-\mu\wi{K}^{'}_{0,1}&\mu\la_{1}(T_{1,2}-T_{1,1})
                                                              &\mu(\la_{1}K^{'}_{1,2}-K^{'}_{1,1})\\
\end{matrix}\right)
 \enn
which is compact since all its entries are compact.
Then the system $(\ref{u})-(\ref{phi1})$ can be rewritten in the abbreviated form
 \be\label{matrix-form}
 (I+A)U=R,
 \en
where $U=(u,v,w,\psi_0,\phi_0,\psi_1,\phi_1)^T$, $R=(0,0,0,\la f,-\la g,\mu p,-\mu q)^T$
and $I$ is the identity operator.
Thus, the Riesz-Fredholm theory is applicable to establish the existence of solutions
to the system (\ref{matrix-form}) if we can prove the uniqueness of solutions to the system.
To this end, let $U$ be a solution of the homogeneous system corresponding to (\ref{matrix-form})
(that is, the system (\ref{matrix-form}) with $R=0$). Then it is enough to show that $U=0$.

We first prove that $\psi_{0}=\phi_{0}=0$ on $S_{0}$. From the system (\ref{matrix-form}) or
(\ref{u})-(\ref{phi1}) with $\la f=-\la g=0$ (since $R=0$)
it is known that $u, v, w$ defined in (\ref{u})-(\ref{w}) satisfy the problem
$(\ref{HE0})-(\ref{rc})$ with $f=g=0$. Thus, by the uniqueness Theorem \ref{uni.direct},
$u=0$ in $\Om_0$, $v=0$ in $\Om_1$ and $w=0$ in $\Om_2$.

Now define
 \ben
\wi{v}&:=&\wi{K}_{0,1}\psi_{0}+\wi{S}_{0,1}\phi_{0}+\la_1\wi{K}_{1,1}\psi_{1}
          +\wi{S}_{1,1}\phi_{1}\;\quad\mbox{in}\;\;\Om_{0},\\
\wi{u}&:=&-\wi{K}_{0,0}\psi_{0}-\frac{1}{\la_{0}}\wi{S}_{0,0}\phi_{0}\;\hspace{3.15cm}\mbox{in}\;\;\Om.
 \enn
Then by the jump relations for single- and double-layer potentials
we have
 \be\label{htbc1}
 \wi{v}-v=\psi_{0},&&\;\;\;\frac{1}{\la_0}u+\wi{u}=\psi_{0}\hspace{1.55cm}\mbox{on}\;\;S_0,\\ \label{htbc2}
 \frac{\pa \wi{v}}{\pa\nu}-\frac{\pa v}{\pa\nu}=-\phi_{0},&&\;\;\;
  \frac{\pa u}{\pa\nu}+\la_0\frac{\pa \wi{u}}{\pa\nu}=-\phi_{0}\qquad\mbox{on}\;\;S_0.
 \en
Thus, $\wi{v}$ and $\wi{u}$ solve the homogeneous transmission problem
 \ben
 \Delta \wi{v}+k_{0}^{2}\wi{v}=0\;\;\;\; \mbox{in}\;\;\Om_{0},\qquad
 \Delta \wi{u}+k_{1}^{2}\wi{u}=0\;\;\;\; \mbox{in}\;\;\Om
 \enn
with the transmission conditions
 \ben
 \wi{v}-\wi{u}=0,\;\quad
 \frac{\pa \wi{v}}{\pa\nu}=\la_0\frac{\pa\wi{u}}{\pa\nu}\qquad\mbox{on}\;\;S_{0}.
 \enn
It is clear that $\wi{v}$ satisfies the radiation condition (\ref{rc}).
By \cite[Lemma 3.40]{CK83}, $\wi{v}=0$ in $\Om_0$ and $\wi{u}=0$ in $\Om$.
Consequently, from the boundary conditions (\ref{htbc1}) and (\ref{htbc2})
we conclude that $\psi_{0}=\phi_{0}=0$ on $S_{0}$.

In a completely similar manner, we can also prove that $\psi_{1}=\phi_{1}=0$ on $S_{1}$.
Thus, we have established the injectivity of the operator $I+A$ and, by
the Riesz-Fredholm theory, $(I+A)^{-1}$ exists and is bounded in $X$.
The estimate (\ref{estimate}) follows from (\ref{matrix-form}).
The proof is thus complete.
\end{proof}

We now establish the following mixed reciprocity relation which is needed for
the inverse problem.

\begin{lemma}\label{Mrr} {\rm( Mixed reciprocity relation.)}
For the scattering of plane waves $u^{i}(\cdot,d)$ with $d\in S$ and
point-sources $\Phi(\cdot,z)$ from the obstacle $\Om_{2}$ we have
\ben
\Phi^{\infty}(\widehat{x},z)=\begin{cases}
      \ds\frac{1}{4\pi}u^{s}(z,-\widehat{x}), &\qquad z\in\Om_{0},\;\widehat{x}\in S,\cr
      \ds\frac{\la_{0}}{4\pi}v(z,-\widehat{x}),&\qquad z\in \Om_{1},\;\widehat{x}\in S.
\end{cases}
\enn
\end{lemma}

\begin{proof}
By Green's second theorem and the radiation condition (\ref{rc}) we get
 \be\label{M1}
\frac{1}{4\pi}\int_{S_{0}}\left(u^{s}(y,z)\frac{\pa u^{s}(y,-\widehat{x})}{\pa\nu(y)}
        -u^{s}(y,-\widehat{x})\frac{\pa u^{s}(y,z)}{\pa\nu(y)}\right)ds(y)=0
 \en
for $z\in\Om_{0}\cup\Om_{1}$ and $\widehat{x}\in S.$
By \cite[Theorem 2.5]{CK}, we have the representation
 \be\label{M2}
\Phi^{\infty}(\widehat{x},z)=\frac{1}{4\pi}\int_{S_{0}}\left(u^{s}_{+}(y;z)
     \frac{\pa e^{-ik_{0}\widehat{x}\cdot y}}{\pa\nu(y)}
     -e^{-ik_{0}\widehat{x}\cdot y}\frac{\pa u^{s}_{+}(y;z)}{\pa\nu(y)}\right)ds(y)
 \en
for $z\in\Om_{0}\cup\Om_{1}$ and $\widehat{x}\in S.$
Adding $(\ref{M1})$ to $(\ref{M2})$ gives
 \ben
 \Phi^{\infty}(\widehat{x},z)=\frac{1}{4\pi}\int_{S_{0}}
 \left(u^{s}(y,z)\frac{\pa u(y,-\widehat{x})}{\pa\nu(y)}
   -u(y,-\widehat{x})\frac{\pa u^{s}(y,z)}{\pa\nu(y)}\right)ds(y)
 \enn
for $z\in\Om_{0}\cup\Om_{1}$ and $\widehat{x}\in S.$

We first consider the case $z\in\Om_{0}$. Use the boundary condition on $S_{0}$
and Green's second theorem to deduce that for $z\in\Om_{0},\;\widehat{x}\in S$,
\be\label{M3}
\Phi^{\infty}(\widehat{x},z)&=&\frac{1}{4\pi}\int_{S_{0}}\left(u^{s}(y,z)
    \frac{\pa u(y,-\widehat{x})}{\pa\nu(y)}
    -u(y,-\widehat{x})\frac{\pa u^{s}(y,z)}{\pa\nu(y)}\right)ds(y)\cr
 &=&\la_{0}\frac{1}{4\pi}\int_{S_{1}}\left((v(y,z)-\Phi_{0}(y,z))
    \frac{\pa v(y,-\widehat{x})}{\pa\nu(y)}
    -v(y,-\widehat{x})\frac{\pa (v(y,z)-\Phi_{0}(y,z))}{\pa\nu(y)}\right)ds(y)\cr
 &&+\la_{0}\frac{1}{4\pi}\int_{\Om_{1}}\left((k_{1}^{2}-k_{0}^{2})
    \Phi_{0}(y,z)v(y,-\widehat{x})\right)dy\cr
 &&+(1-\la_{0})\frac{1}{4\pi}\int_{S_{0}}\left(v(y,-\widehat{x})
    \frac{\pa\Phi_{0}(y,z)}{\pa\nu(y)}\right)ds(y).
\en
Now Green's second theorem gives that for $z\in\Om_{0}$ and $\widehat{x}\in S$
\be\label{M4}
\frac{1}{4\pi}\int_{S_0}\left(u^{i}(y,-\widehat{x})\frac{\pa
\Phi_{0}(z,y)}{\pa\nu(y)}
   -\Phi_{0}(z,y)\frac{\pa u^{i}(y,-\widehat{x})}{\pa\nu(y)}\right)ds(y)=0.
\en
From Green's formula (see Theorem 2.4 in \cite{CK}) it follows that
\be\label{M5}
\frac{1}{4\pi}u^{s}(z,-\widehat{x})=\frac{1}{4\pi}\int_{S_{0}}\left(u^{s}(y,-\widehat{x})
     \frac{\pa\Phi_{0}(z,y)}{\pa\nu(y)}
     -\Phi_{0}(z,y)\frac{\pa u^s(y,-\widehat{x})}{\pa\nu(y)}\right)ds(y)
\en
for $z\in\Om_0$ and $\widehat{x}\in S.$ Adding (\ref{M4}) to the
equation (\ref{M5}) we deduce, with the help of the transmission
condition on $S_0$ and Green's second theorem, that for $z\in\Om_{0}$ and $\widehat{x}\in S,$
\be\label{M6}
\frac{1}{4\pi}u^{s}(z,-\widehat{x})&=&\frac{1}{4\pi}\int_{S_{0}}\left(u(y,-\widehat{x})
  \frac{\pa\Phi_{0}(z,y)}{\pa\nu(y)}-\Phi_{0}(z,y)
    \frac{\pa u(y,-\widehat{x})}{\pa\nu(y)}\right)ds(y)\cr
 &=&\la_{0}\frac{1}{4\pi}\int_{S_{1}}\left(v(y,-\widehat{x})\frac{\pa \Phi_{0}(z,y)}{\pa\nu(y)}
  -\Phi_{0}(z,y)\frac{\pa v(y,-\widehat{x})}{\pa\nu(y)}\right)ds(y)\cr
 &&+\la_{0}\frac{1}{4\pi}\int_{\Om_{1}}\left((k_{1}^{2}-k_{0}^{2})
   \Phi_{0}(y,z)v(y,-\widehat{x})\right)dy\cr
 &&+(1-\la_{0})\frac{1}{4\pi}\int_{S_{0}}\left(v(y,-\widehat{x})
   \frac{\pa \Phi_{0}(y,z)}{\pa\nu(y)}\right)ds(y)
\en
By (\ref{M3}) and (\ref{M6}) together with the transmission condition on $S_{1}$
and the equation (\ref{HE2}) we deduce that
 \ben
\Phi^{\infty}(\widehat{x},z)-\frac{1}{4\pi}u^{s}(z,-\widehat{x})
&=&\la_{0}\frac{1}{4\pi}\int_{S_{1}}\left(v(y,z)\frac{\pa v(y,-\widehat{x})}{\pa\nu(y)}
    -v(y,-\widehat{x})\frac{\pa v(y,z)}{\pa\nu(y)}\right)ds(y)\cr
&=&\la_{0}\la_{1}\frac{1}{4\pi}\int_{S_{1}}\left(w(y,z)\frac{\pa w(y,-\widehat{x})}{\pa\nu(y)}
    -w(y,-\widehat{x})\frac{\pa w(y,z)}{\pa\nu(y)}\right)ds(y)\cr
&=&\la_{0}\la_{1}\frac{1}{4\pi}\int_{\Om_{2}}(w(y,z)\Delta w(y,-\widehat{x})
    -w(y,-\widehat{x})\Delta w(y,z))ds(y)=0.
 \enn

We now consider the case $z\in\Om_1$. Using the transmission condition on $S_0$ and $S_{1}$, the
equations $(\ref{HE1})$ and $(\ref{HE2})$ and Green's second theorem, we obtain that
 \ben
\Phi^{\infty}(\widehat{x},z)&=&\frac{1}{4\pi}\int_{S_{0}}\left(u^{s}(y,z)
  \frac{\pa u(y,-\widehat{x})}{\pa\nu(y)}-u(y,-\widehat{x})
  \frac{\pa u^{s}(y,z)}{\pa\nu(y)}\right)ds(y)\cr
&=&\frac{\la_{0}}{4\pi}\int_{S_{0}}\left((v(y,z)+\Phi_{1}(y,z))
   \frac{\pa v(y,-\widehat{x})}{\pa\nu(y)}-v(y,-\widehat{x})
  \frac{\pa (v(y,z)+\Phi_{1}(y,z))}{\pa\nu(y)}\right)ds(y)\cr
&=&\frac{\la_{0}}{4\pi}\int_{S_{0}}\left((v(y,z)+\Phi_{1}(y,z))
   \frac{\pa v(y,-\widehat{x})}{\pa\nu(y)}-v(y,-\widehat{x})
   \frac{\pa (v(y,z)+\Phi_{1}(y,z))}{\pa\nu(y)}\right)ds(y)\cr
&=&\frac{\la_{0}}{4\pi}\int_{S_{0}}\left(\Phi_{1}(y,z)
   \frac{\pa v(y,-\widehat{x})}{\pa\nu(y)}-v(y,-\widehat{x})
   \frac{\pa \Phi_{1}(y,z)}{\pa\nu(y)}\right)ds(y)\cr
&&+\frac{\la_{0}}{4\pi}\int_{S_{1}}\left(v(y,z)
   \frac{\pa v(y,-\widehat{x})}{\pa\nu(y)}-v(y,-\widehat{x})
   \frac{\pa v(y,z)}{\pa\nu(y)}\right)ds(y)\cr
&=&\frac{\la_{0}}{4\pi}\int_{S_{0}}\left(\Phi_{1}(y,z)
   \frac{\pa v(y,-\widehat{x})}{\pa\nu(y)}-v(y,-\widehat{x})
   \frac{\pa \Phi_{1}(y,z)}{\pa\nu(y)}\right)ds(y)\cr
&&-\frac{\la_{0}}{4\pi}\int_{S_{1}}\left(\Phi_{1}(y,z)
   \frac{\pa v(y,-\widehat{x})}{\pa\nu(y)}-v(y,-\widehat{x})
   \frac{\pa \Phi_{1}(y,z)}{\pa\nu(y)}\right)ds(y)\cr
&=&\frac{\la_{0}}{4\pi}v(z,-\widehat{x}),
 \enn
where use has been made of Green's formula (see Theorem 2.1 in \cite{CK}).
\end{proof}

Denote by $D$ any of $\ov{\Om_{R}},\;\ov{\Om_{1}},\;\ov{\Om_{2}},$ $S_{0},$ $S_{1}$.
Let $x^{\ast}\in D$ be an arbitrarily fixed point and let us introduce the space $C_{0}(D)$
which consists of all continuous functions $h\in C(D\ba\{x^{\ast}\})$ with the property that
 \ben
\lim_{x\rightarrow x^{\ast}}|(x-x^{\ast})h(x)|
 \enn
exists. It can be easily verify that $C_{0}(D)$ is a Banach space equipped with
the weighted maximum norm
 \ben
 \|h\|_{0,\;D}:=\sup_{x\neq x^{\ast},\; x\in D}|(x-x^{\ast})h(x)|.
 \enn

To prove the unique determination of $S_{0}$ in the inverse problem in the next section,
we need to consider the the boundary value problem $(\ref{HE0})-(\ref{rc})$ with
$\ds f=-\Phi_0(\cdot,z_0)|_{S_0},\;g=-{\pa\Phi_0(\cdot,z_0)}/{\pa\nu}|_{S_0},\;p=0,\;q=0,$
where $z_{0}\in\Om_{0},$ and investigate the behavior of the solution $v$ on some part of $S_0$.

\begin{lemma}\label{vapriori}
For given functions $f\in C^{1,\al}(S_{0})$ and $g\in C^{0,\al}(S_{0})$,
assume that $u\in C^{2}(\Om_{0})\cap C^{1,\al}(\ov{\Om_{0}})$,
$v\in C^{2}(\Om_{1})\cap C^{1,\al}(\ov{\Om_{1}})$ and
$w\in C^{2}(\Om_{2})\cap C^{1,\al}(\ov{\Om_{2}})$
are the solution of the problem $(\ref{HE0})-(\ref{rc}).$
Let $x^{\ast}\in S_{0}$ and let $B_{1},$ $B_{2}$ $(B_{2}\cap \ov{\Om_{2}}=\emptyset)$
be two small balls with center $x^{\ast}$
and radii $r_{1},\;r_{2}$, respectively, satisfying that $r_1<r_2$.
Then there exists a constant $C>0$ such that
 \be
 \|v\|_{\infty,\; S_{0}\ba B_{2}}+\left\|\frac{\pa v}{\pa\nu}\right\|_{\infty,\; S_{0}\ba B_{2}}
 \leq C(\|f\|_{0,\;S_{0}}+\|g\|_{0,\;S_{0}}+\|f\|_{1,\;\al,\;S_{0}\ba B_{1}}
   +\|g\|_{0,\;\al,\;S_{0}\ba B_{1}})
 \en
\end{lemma}

\begin{proof}
We consider the system (\ref{matrix-form}) of integral equations derived from
the boundary value problem $(\ref{HE0})-(\ref{rc})$.
In addition to the space $X$, we also consider the weighted spaces
$C_{0}:=C_{0}(\ov{\Om_{R}})\times C_{0}(\ov{\Om_{1}})\times C^{1,\;\al}(\ov{\Om_{2}})\times
C_{0}(S_{0})\times C_{0}(S_{0})\times C^{1,\al}(S_{1})\times C^{0,\al}(S_{1})$.
The matrix operator $A$ is also compact in $C_0$ since all entries of $A$ are
compact (see \cite{KK93,KP}). From the proof of Theorem \ref{well-posedness},
we know that the operator $I+A$ has a trivial null space in $X$. Therefore,
by the Fredholm alternative applied to the dual system $\langle X,C_{0}\rangle$
with the $L^{2}$ bilinear form, the adjoint operator $I+A^\prime$
has a trivial null space in $C_{0}$. By the Fredholm alternative again, but now
applied to the dual system $\langle C_{0}, C_{0}\rangle$ with the $L^{2}$ bilinear form,
the operator $I+A$ also has a trivial null space in $C_{0}$.
Hence, by the Riesz-Fredholm theory, the system (\ref{matrix-form}) is also uniquely
solvable in $C_{0}$, and the solution depends continuously on the right-hand side:
 \be\no
 \|U\|_{0}&:=&\|u\|_{0,\;\ov{\Om}_{R}}+\|v\|_{0,\;\ov{\Om}_{1}}+\|w\|_{1,\;\al,\;\ov{\Om}_{2}}
 +\|\psi_{0}\|_{0,\;S_{0}}+\|\phi_{0}\|_{0,\;S_{0}}
 +\|\psi_{1}\|_{1,\;\al,\;S_{1}}+\|\phi_{1}\|_{0,\;\al,\;S_{1}}\\ \label{C0es}
 &\leq& C(\|f\|_{0,\;S_{0}}+\|g\|_{0,\;S_{0}}).
 \en
In particular, this implies that
 \be
 \label{ves}\|v\|_{\infty,\;S_{0}\ba B_{2}}\leq C(\|f\|_{0,\;S_{0}}+\|g\|_{0,\;S_{0}}).
 \en
Before proceeding to estimate ${\pa v}/{\pa\nu}$ we establish the following estimate
in the spaces of H\"{o}lder continuous functions for $(u,v,w,\psi_0,\phi_0,\psi_1,\phi_1)$:
 \be\no
 \|U\|_{1,\;\al,\;3}&:=&\|u\|_{0,\;\al,\;\ov{\Om}_R\ba B_3}+\|v\|_{0,\;\al,\;\ov{\Om}_1\ba B_3}
       +\|w\|_{1,\;\al,\;\ov{\Om}_2}\\ \no
 &&+\|\psi_0\|_{1,\;\al,\;S_0\ba B_3}+\|\phi_0\|_{0,\;\al,\;S_0\ba B_3}
   +\|\psi_{1}\|_{1,\;\al,\;S_{1}}+\|\phi_{1}\|_{0,\;\al,\;S_{1}}\\ \label{7-tuple}
 &\leq& C(\|f\|_{0,\;S_{0}}+\|g\|_{0,\;S_{0}}
   +\|f\|_{1,\;\al,\;S_{0}\ba B_{1}}+\|g\|_{0,\;\al,\;S_{0}\ba B_{1}})
 \en
where $B_{3}$ is a ball of radius $r_{3}$ and centered at $x^*$ with $r_{1}<r_{3}<r_{2}$.

We now choose a function $\rho_{1}\in C^{2}(S_{0})$ such that $\rho_1(x)=0$ for $x\in S_0\ba B_2$
and $\rho_{1}(x)=1$ in the neighborhood of $B_{3}$.
We also choose another function $\rho_{2}\in C^{2}(S_{0})$ such that $\rho_{2}(x)=1$
for $x\in S_{0}\ba B_{3}$ and $\rho_{2}(x)=0$ in the neighborhood of $B_{1}$.

Split $U$ up in the form
\ben
 U=\left (\begin{matrix}\rho_{1}u\\ \rho_{1}v\\ w\\ \rho_{1}\psi_{0}\\
                     \rho_{1}\phi_{0}\\ \psi_{1}\\ \phi_{1}\end{matrix} \right )
   +\left (\begin{matrix}(1-\rho_{1})u\\ (1-\rho_{1})v\\ w\\ (1-\rho_{1})\psi_{0}\\
                     (1-\rho_{1})\phi_{0}\\ \psi_{1}\\ \phi_{1}\end{matrix} \right )
   :=U_{\rho_{1}}+U_{1-\rho_{1}}
\enn
and for a matrix $W$ use $W_{\rho}$ to denote the same matrix but with its first, second,
fourth and fifth rows multiplied by $\rho(x)$. Then it follows from (\ref{matrix-form}) that
\be\label{rhoU}
 U_{\rho_{2}}=R_{\rho_{2}}-A_{\rho_{2}}U_{\rho_{1}}-A_{\rho_{2}}U_{1-\rho_{1}}.
\en
The mapping operator $U\rightarrow A_{\rho_{2}}U_{\rho_{1}}$ is bounded from $C_0$ into $X$ since
its kernel vanishes in a neighborhood of the diagonal $x=y$. Furthermore, we apply Theorems
2.30 and 2.31 in \cite{CK83} and the analogous results for
$T_{0,1}, K^{'}_{0,1}, \wi{T}_{1,1}, \wi{K}^{'}_{1,1}$ to obtain that
 \ben
 \|A_{\rho_{2}}U_{1-\rho_{1}}\|_{0,\;\al}\leq C\|AU_{1-\rho_{1}}\|_{0,\;\al}
  \leq C\|U_{1-\rho_{1}}\|_{\infty}\leq C\|U\|_{0},
 \enn
where the norms $\|U\|_{0,\;\al}$ and $\|U\|_{\infty}$ are defined as follows:
the first, second, fourth and fifth components of $U$ are defined by the corresponding norms
and its third, sixth and seventh components are equipped with the
$C^{1,\;\al}(\ov{\Om_{2}}),\; C^{1,\;\al}(S_{1})$ and $C^{0,\;\al}(S_{1})$ norms, respectively.
From (\ref{rhoU}) and (\ref{C0es}) it is derived that
 \be\no
 \|U\|_{0,\;\al,\;3}&:=&\|u\|_{0,\;\al,\;\ov{\Om}_R\ba B_3}+\|v\|_{0,\;\al,\;\ov{\Om}_1\ba B_3}
   +\|w\|_{1,\;\al,\;\ov{\Om}_2}\\ \no
 &&+\|\psi_{0}\|_{0,\;\al,\;S_{0}\ba B_{3}}+\|\phi_{0}\|_{0,\;\al,\;S_{0}\ba B_{3}}
   +\|\psi_{1}\|_{1,\;\al,\;S_{1}}+\|\phi_{1}\|_{0,\;\al,\;S_{1}}\\ \no
 &\leq& C\|U_{\rho_{2}}\|_{0,\;\al}\leq C(\|R_{\rho_{2}}\|_{0,\;\al}+\|U\|_{0})\\ \label{U0al3}
 &\leq& C(\|f\|_{0,\;\al,\;S_{0}\ba B_{1}}+\|g\|_{0,\;\al,\;S_{0}\ba B_{1}}
   +\|f\|_{0,\;S_{0}}+\|g\|_{0,\;S_{0}}).
 \en

It remains to estimate $\|\psi_{0}\|_{1,\;\al,\;S_{0}\ba B_{3}}$.
Multiplying (\ref{psi0}) by $\rho_{2}(x)$ we obtain, on using (\ref{rhoU}) and noting
the fact that the integral operators mapping $C^{0,\al}$-functions
into $C^{1,\al}$-functions are bounded, that
 \be
 \no\|\psi_{0}\|_{1,\;\al,\;S_{0}\ba B_{3}}&\leq&\|\rho_{2}\psi_{0}\|_{1,\;\al,\;S_{0}}\\
 \no&\leq& C(\|\rho_{2}(\la_0K_{0,0}-K_{0,1})\psi_{0}\|_{1,\al}
  +\|\rho_{2}(S_{0,0}-S_{0,1})\phi_{0}\|_{1,\al}\\
 \no&&+\|\rho_{2}\wi{K}_{1,1}\psi_{1}\|_{1,\al}
  +\|\rho_{2}\wi{S}_{1,1}\phi_{1}\|_{1,\al}+\|\rho_{2}f\|_{1,\al})\\
 \no&\leq& C(\|U\|_{0}+\|(1-\rho_{1})U\|_{0,\;\al}+\|\rho_{2}f\|_{1,\al})\\
 \label{psi01al}&\leq& C(\|U\|_{0}+\|U\|_{0,\;\al,\;3}+\|f\|_{1,\;\al,\;S_{0}\ba B_{1}})
\en
where we have used the fact that $\psi_0=\rho_1\psi_0+(1-\rho_1)\psi_0$ and
$\phi_0=\rho_1\phi_0+(1-\rho_1)\phi_0$ in order to get the third inequality.
Combining (\ref{C0es}) and (\ref{U0al3})-(\ref{psi01al}) yields the desired estimate (\ref{7-tuple}).

We now estimate $\ds\left\|\frac{\pa v}{\pa\nu}\right\|_{0,\;\al,\;S_{0}\ba B_{2}}$.
From (\ref{v}) and the jump relation we have that on $S_0,$
\be\label{pavS0}
\frac{\pa v}{\pa\nu}=\frac{1}{2}\phi_{0}+T_{0,1}\psi_{0}+K^{'}_{0,1}\phi_{0}
   +\la_{1}\wi{T}_{1,1}\psi_{1}+\wi{K}^{'}_{1,1}\phi_{1}.
\en
Writing $\psi_{0}=\rho_{1}\psi_{0}+(1-\rho_{1})\psi_{0}$ and
$\phi_{0}=\rho_{1}\phi_{0}+(1-\rho_{1})\phi_{0}$, it follows from (\ref{pavS0}) that
\ben
\left\|\frac{\pa v}{\pa\nu}\right\|_{0,\;\al,\;S_{0}\ba B_{2}}
&\leq& \left\|(1-\rho_{1})\frac{\pa v}{\pa\nu}\right\|_{0,\;\al,\;S_{0}}\\
&\leq& C(\|U\|_0+\|(1-\rho_1)\psi_0\|_{1,\;\al,\;S_0}+\|(1-\rho_1)\phi_0\|_{0,\;\al,\;S_0}\\
&&     +\|\psi_{1}\|_{1,\;\al,\;S_{1}}+\|\phi_{1}\|_{0,\;\al,\;S_{1}})\\
&\leq& C(\|U\|_{0}+\|U\|_{1,\;\al,\;3}).
\enn
Combining this with (\ref{C0es}) and (\ref{7-tuple}) gives
\be\label{paves}
\left\|\frac{\pa v}{\pa\nu}\right\|_{0,\;\al,\;S_{0}\ba B_{2}}
\leq C(\|f\|_{C_{0}(S_{0})}+\|g\|_{C_{0}(S_{0})}
 +\|f\|_{1,\al,S_{0}\ba B_{1}}+\|g\|_{0,\al,S_{0}\ba B_{1}}).
\en
This completes the proof.
\end{proof}

To prove the unique determination of $S_{1}$, we consider the boundary value problem
$(\ref{HE0})-(\ref{rc})$ in the case when the incident field is the point source
$\Phi_1(\cdot,z_1)$ with $z_1\in\Om_1$, that is,
$\ds f=u^i|_{S_0},\;g={\pa u^i}/{\pa\nu}|_{S_0},\;p=-u^i|_{S_1},\;q=-{\pa u^i}/{\pa\nu}|_{S_1}$.
Arguing similarly as in the proof of Lemma \ref{vapriori}, we can obtain the following
a priori estimate of the solution $w$ on some part of $S_1$.

\begin{lemma}\label{wapriori}
Given four functions $f\in C^{1,\al}(S_{0})$, $g\in C^{0,\al}(S_{0})$,
$p\in C^{1,\al}(S_{1})$ and $q\in C^{0,\al}(S_{1})$,
let $u\in C^{2}(\Om_{0})\cap C^{1,\al}(\ov{\Om}_0)$,
$v\in C^{2}(\Om_{1})\cap C^{1,\al}(\ov{\Om}_1)$ and
$w\in C^{2}(\Om_{2})\cap C^{1,\al}(\ov{\Om}_2)$
be the solution of the problem $(\ref{HE0})-(\ref{rc}).$
Let $x^{\ast}\in S_{1}$ and let $B_{1},$ $B_{2}$ $(B_2\cap\ov{\Om}_0=\emptyset)$
be two small balls with center $x^{\ast}$
and radii $r_{1},\;r_{2}$ ($r_1<r_2$), respectively.
Then there exists a constant $C>0$ such that
 \ben
 \|w\|_{\infty,\; S_{1}\ba B_{2}}+\left\|\frac{\pa w}{\pa\nu}\right\|_{\infty,\; S_{1}\ba B_{2}}
 &\leq& C(\|f\|_{1,\;\al,\;S_{0}}+\|g\|_{0,\;\al,\;S_{0}}\\
 &&+\|p\|_{0,\;S_{1}}+\|q\|_{0,\;S_{1}}+\|p\|_{1,\;\al,\;S_{1}\ba B_{1}}
   +\|q\|_{0,\;\al,\;S_{1}\ba B_{1}}).
 \enn
\end{lemma}

\section{Unique determination of the interfaces $S_0$ and $S_1$}\label{sec3}
\setcounter{equation}{0}

Following the ideas of Kirsch and Kress \cite{KK93} for transmission problems in
a homogeneous medium, Kirsch and P\"{a}iv\"{a}rinta \cite{KP} for transmission problems
in an inhomogeneous medium and Liu and Zhang \cite{LZ2} for impenetrable obstacle scattering
in a piecewise homogeneous medium, we prove, in this section, that the interfaces
$S_{0}$ and $S_{1}$ can be uniquely determined by the far field pattern.
To do this, we need the following four lemmas, in which $\Om=\R^3\ba\ov{\Om}_0$
so that $\Om_{2}\subset\Om$ and $\wi{\Om}=\R^3\ba\ov{\wi{\Om}}_0$
for some domain $\wi{\Om}_0$ with interface $\wi{S}_0=\pa\wi{\Om}_0\cap\pa\wi{\Om}$
and domain $\wi{\Om}_2\subset\wi{\Om}$.

\begin{lemma}\label{lem3.1}
For $\Om_2\subset\Om,\;\wi{\Om}_2\subset\wi{\Om}$ let $G_{0}$ be the unbounded component
of $\R^3\setminus(\ov{\Om\cup\wi{\Om}})$ and let
$u^{\infty}(\hat{x},d)=\wi{u}^{\infty}(\hat{x},d)$ for all $\hat{x},\;d\in S$,
where $\wi{u}^{\infty}(\hat{x},d)$ is the far field pattern of the scattered field
$\wi{u}^s(x,d)$ corresponding to the obstacle $\wi{\Om}_2$, the interface $\wi{S}_0$
and the same incident plane wave $u^i(x,d)$.
For $z\in G_{0}$ let $(u^{s}(\cdot,z),v(\cdot,z),w(\cdot,z))$ be the unique solution of the problem
\be
\label{lHE0}\Delta u^{s}+k_0^2u^s=0&&\;{\rm in}\;\Om_0\setminus\{z\},\\
\label{lHE1}\Delta v+k_1^2v=0&&\;{\rm in}\;\Om_1,\\
\label{lHE2}\Delta w+k_2^2nw=0&&\;{\rm in}\;\Om_2,\\
\label{ltbc0}u^s-v=-\Phi_{0}(\cdot,z),\;\;\frac{\pa u^s}{\pa\nu}-\la_0\frac{\pa v}{\pa\nu}
            =-\frac{\pa\Phi_{0}(\cdot,z)}{\pa\nu}&&\;{\rm on}\;S_0,\\
\label{ltbc1}v-w=0,\;\;\frac{\pa v}{\pa\nu}=\la_1\frac{\pa w}{\pa\nu}
            &&\;{\rm on}\;S_1,\\
\label{lrc}\lim_{r\rightarrow\infty}r\left(\frac{\pa u^{s}}{\pa r}-ik_{0}u^{s}\right)=0.
\en
Assume that $(\wi{u}^{s}(\cdot,z),\wi{v}(\cdot,z),\wi{w}(\cdot,z))$ is the unique solution
of the problem $(\ref{lHE0})-(\ref{lrc})$ with $\Om_0,\Om_1,\Om_2,S_0,S_1,n$ replaced by
$\wi{\Om_0},\wi{\Om_1},\wi{\Om_2},\wi{S_0},\wi{S_1},\wi{n}$, respectively.
Then we have
\ben\label{3.22}
u^s(x,z)=\wi{u}^s(x,z),\qquad x\in\ov{G_{0}}.
\enn
\end{lemma}

\begin{remark}\label{r3.2} {\rm
By Theorem \ref{well-posedness}, the problem (\ref{lHE0})-(\ref{lrc}) has
a unique solution.
}
\end{remark}

\begin{proof}
By Rellich's lemma \cite{CK}, the assumption $u^{\infty}(\hat{x},d)
=\wi{u}^{\infty}(\hat{x},d)$ for all $\hat{x},\;d\in S$ implies that
 \ben
u^s(x,d)=\wi{u}^{s}(x,d),\qquad x\in G_{0},\;d\in S.
 \enn
Then, for the far field pattern corresponding to incident point-sources
we have by Lemma \ref{Mrr} that
 \ben
\Phi^{\infty}(d,z)=\wi{\Phi}^{\infty}(d,z),\qquad z\in G_{0},\;d\in S.
 \enn
Thus, Rellich's lemma \cite{CK} implies that
 \ben
u^{s}(x,z)=\wi{u}^{s}(x,z),\qquad x\in\ov{G_{0}}.
 \enn
\end{proof}

Arguing similarly as in the proof of Lemma \ref{lem3.1}, we have the
corresponding result for the point source in the bounded homogeneous medium.

\begin{lemma}\label{lem3.2}
For $\Om_2\subset\Om,\;\wi{\Om}_2\subset\Om$
let $G_1:=\Om\setminus(\ov{\Om_2\cup\wi{\Om_2}})$ and let
$u^{\infty}(\hat{x},d)=\wi{u}^{\infty}(\hat{x},d)$ for all $\hat{x},\;d\in S$
where $\wi{u}^{\infty}(\hat{x},d)$ is the far field pattern of the scattered field
$\wi{u}^s(x,d)$ corresponding to the obstacle $\wi{\Om}_2$, the interface $\wi{S}_1$
and the same incident plane wave $u^i(x,d)$.
For $z\in G_{1}$ let $(u^{s}(\cdot,z),v(\cdot,z),w(\cdot,z))$
be the unique solution of the problem
\be
\label{2HE0}\Delta u^{s}+k_0^2u^s=0&&\;{\rm in}\;\;\Om_0,\\
\label{2HE1}\Delta v+k_1^2v=0&&\;{\rm in}\;\;\Om_1\setminus\{z\},\\
\label{2HE2}\Delta w+k_2^2nw=0&&\;{\rm in}\;\;\Om_2,\\
\label{2tbc0}u^s-v=\Phi_{1}(x,z),\;\;\frac{\pa u^s}{\pa\nu}-\la_0\frac{\pa v}{\pa\nu}
            =\frac{\pa\Phi_{1}(x,z)}{\pa\nu}&&\;{\rm on}\;\;S_0,\\
\label{2tbc1}v-w=-\Phi_{1}(x,z),\;\;\frac{\pa v}{\pa\nu}-\la_1\frac{\pa w}{\pa\nu}
                =-\frac{\pa\Phi_{1}(x,z)}{\pa\nu}
            &&\;{\rm on}\;\;S_1,\\
\label{2rc}\lim_{r\rightarrow\infty}r\left(\frac{\pa u^{s}}{\pa r}-ik_{0}u^{s}\right)=0.
\en
Assume that $(\wi{u}^{s}(\cdot,z),\wi{v}(\cdot,z),\wi{w}(\cdot,z))$ is the unique
solution of the problem $(\ref{2HE0})-(\ref{2rc})$ with $\Om_1,\Om_2,S_1,n$
replaced by $\wi{\Om_1},\wi{\Om_2},\wi{S_1},\wi{n}$, respectively.
Then we have
\ben\label{3.23}
v(x,z)=v(x,z),\qquad x\in\ov{G_{1}}.
\enn
\end{lemma}


\begin{lemma}\label{lem3.3}
Assume that $f\in L^{2}(\Om_{1}),p\in C(S_{0}),q\in C(S_{1})$, $g,\eta\in C(S_{0})$ with
$\eta\neq0$ and $\eta\leq0$ on $S_{0}$. Then the following problem has a unique solution
$v\in C^{2}(\Om_{1})\cap C(\ov{\Om_{1}})$ and $w\in C^{2}(\Om_{2})\cap C(\ov{\Om_{2}})$:
 \be
 \label{vOm1}\Delta v+k_{1}^{2}v=f&&\qquad {\rm in}\;\;\Om_{1},\\
 \label{wOm2}\Delta w+k_{2}^{2}nw=0&&\qquad {\rm in}\;\;\Om_{2},\\
 \label{vbg}\frac{\pa v}{\pa\nu}+i\eta v=g&&\qquad{\rm on}\;\;S_0,\\
 \label{vwtr}v-w=p,\;\;\;\frac{\pa v}{\pa\nu}-\la_{1}\frac{\pa w}{\pa\nu}=q&&\qquad{\rm on}\;\; S_{1}.
 \en
Furthermore, there exists a constant $C>0$ such that
\ben
\|v\|_{\infty,\;\ov{\Om}_1}
\leq C(\|f\|_{L^2(\Om_1)}+\|g\|_{\infty,\;S_0}+\|p\|_{\infty,\;S_1}+\|q\|_{\infty,\;S_1}).
\enn
\end{lemma}

\begin{proof}
We first prove the uniqueness result, that is, $v=0$ in $\Om_{0}$, $w=0$ in $\Om_{1}$
if $f=0$ in $\Om_{0}$, $g=0$ on $S_{0}$, $p=q=0$ on $S_{1}$.
With the help of equations (\ref{vOm1})-(\ref{wOm2}) and
boundary conditions (\ref{vbg})-(\ref{vwtr}), we have
\ben
0&=&\int_{\Om_{1}}\left\{(\Delta v+k_{1}^{2}u)\ov{v}\right\}dx\\
 &=&\int_{\Om_{1}}(-|\nabla v|^{2}+k_{1}^{2}|v|^{2})dx
     +\int_{S_{0}}\ov{v}\frac{\pa v}{\pa\nu}ds
     -\int_{S_{1}}\ov{v}\frac{\pa v}{\pa\nu}ds\\
 &=&\int_{\Om_{1}}(-|\nabla v|^{2}+k_{1}^{2}|v|^{2})dx
     -i\int_{S_{0}}\eta|v|^{2}ds-\la_{1}\int_{S_{1}}\ov{w}\frac{\pa w}{\pa\nu}ds\\
 &=&\int_{\Om_{1}}(-|\nabla v|^{2}+k_{1}^{2}|v|^{2})dx
     -i\int_{S_{0}}\eta|v|^{2}ds-\la_{1}\int_{\Om_{2}}(-k_{2}^{2}n|w|^{2}+|\nabla w|^{2})dx.
\enn
Taking the imaginary part of the above equation gives that $v=0$ on some part $\G$
of $S_{0}$ since both $\eta\neq0$ and $\eta\leq0$ on $S_{0}$ and $\la_{1}$ is a
positive number. By the boundary condition (\ref{vbg}) it follows that
$v={\pa v}/{\pa\nu}=0$ on $\G$. Thus, by Holmgren's uniqueness theorem \cite{Kr},
$v=0$ in $\Om_1$. Using the transmission boundary conditions (\ref{vwtr})
and the assumption on $k_{2}^{2}$, we conclude that $w=0$ in $\Om_2.$

We now prove the existence of solutions using the integral equation method.
To this end, introduce the volume potential
\ben
 (V^{\ast}f)(x):=\int_{\Om_{1}}\Phi_{1}(x,y)f(y)dy,\qquad x\in \Om_{1},
\enn
which defines a bounded operator $V^{\ast}:L^{2}(\Om_{1})\rightarrow H^{2}(\Om_{1})$
(see Theorem 8.2 in \cite{CK}).
Now look for a solution in the form
 \be\label{v1}
 v(x)=&-V^{\ast}f+\wi{S}_{0,1}(\phi_{1}+\phi_{2})+\la_1\wi{K}_{1,1}(\psi_{1}+\psi_{2})
      +\wi{S}_{1,1}(\chi_{1}+\chi_{2})\;\;\; {\rm in}\;\;\Om_1,\\ \label{w1}
 w(x)=&V^{\ast}w+\wi{K}_{1,2}(\psi_{1}+\psi_{2})+\wi{S}_{1,2}(\chi_{1}+\chi_{2})
      \hspace{3.5cm}{\rm in}\;\;\Om_2
 \en
with six unknown densities $\phi_1\in C(S_0)$, $\phi_2\in H^{\frac{1}{2}}(S_0)$,
$\psi_1\in C(S_1)$, $\psi_2\in H^{\frac{1}{2}}(S_1)$, $\chi_1\in C(S_1)$,
$\chi_2\in H^{\frac{1}{2}}(S_1)$.
Then from the jump relations we see that the potentials $v,w$ given by (\ref{v1})
and (\ref{w1}) solve the boundary value problem $(\ref{vOm1})-(\ref{vwtr})$
provided the six densities satisfy the following system of integral equations:
 \ben
\label{3phi1}\phi_{1}+2(K^{'}_{0,1}+i\eta S_{0,1})\phi_{1}
    +2\la_{1}(\wi{T}_{1,1}+i\eta\wi{K}_{1,1})(\psi_{1}+\psi_{2})+2(\wi{K}^{'}_{1,1}
    +i\eta\wi{S}_{1,1})(\chi_{1}+\chi_{2})=2g&\mbox{on}\;S_{0},\\
\label{3phi2}\phi_{2}+2(K^{'}_{0,1}+i\eta S_{0,1})\phi_{2}
    =2\left(\frac{\pa }{\pa\nu}+i\eta\right)(V^{\ast}f)&\mbox{on}\;S_{0},\\
\label{3psi1}\psi_{1}-\mu V^{\ast}w+\mu\wi{S}_{0,1}(\phi_{1}+\phi_{2})
    +\mu(\la_{1}K_{1,1}-K_{1,2})\psi_{1}
    +\mu(S_{1,1}-S_{1,2})\chi_{1}=\mu p&\mbox{on}\;S_{1},\\
\label{3psi2}\psi_{2}+\mu(\la_{1}K_{1,1}-K_{1,2})\psi_{2}
    +\mu(S_{1,1}-S_{1,2})\chi_{2}=\mu V^{\ast}f&\mbox{on}\;S_{1},\\
\label{3chi1}\chi_{1}+\mu\la_{1}\frac{\pa}{\pa\nu}V^{\ast}w
    -\mu\wi{K}^{'}_{0,1}(\phi_{1}+\phi_{2})+\mu\la_{1}(T_{1,2}-T_{1,1})\psi_{1}
    +\mu(\la_{1}K^{'}_{1,2}-K^{'}_{1,1})\chi_{1}=-\mu q&\mbox{on}\;S_{1},\\
\label{3chi2}\chi_{2}+\mu(\la_{1}K^{'}_{1,2}-K^{'}_{1,1})\chi_{2}
    +\mu\la_1(T_{1,2}-T_{1,1})\psi_2=-\mu\frac{\pa}{\pa\nu}(V^{\ast}f)&\mbox{on}\;S_1
 \enn
where $\mu={2}/({\la_{1}+1})$.
This system is well defined since both the single and double layer potentials with
$H^{\frac12}$-density are continuous and the restriction to the boundary of
the derivative of the volume potential with $L^{2}$-density is in $H^{\frac12}$.
The trick to divide the density in a layer potential into a continuous
and $H^{\frac12}$-part has been used previously by Kirsch and P\"{a}iv\"{a}rinta \cite{KP}.


Precisely, we seek a solution $(v,w,\phi_1,\phi_2,\chi_1,\chi_2,\varphi_1,\varphi_2)\in
Y:=C(\ov{\Om_1})\times C(\ov{\Om_2})\times C(S_0)\times H^{\frac12}(S_0)\times C(S_1)\times
H^{\frac12}(S_1)\times C(S_1)\times H^{\frac12}(S_1)$ to the above system of eight integral
equations. Standard arguments using the uniqueness of solutions to the problem show that
this system admits at most one solution in $Y$. Therefore, the Riesz-Fredholm theorem
is applicable and yields the existence of solutions to the boundary value problem
with the estimate
 \ben
 \|v\|_{\infty,\;\ov{\Om}_1}&\leq&C(\|V^{\ast}f\|_{\infty,\ov{\Om}_1}+\|g\|_{\infty, S_0}
 +\left\|\left(\frac{\pa}{\pa\nu}+i\eta\right)(V^{\ast}f)\right\|_{H^{\frac12}(S_0\cup S_1)}\\
 &&+\|p\|_{\infty,\;S_1}+\|q\|_{\infty,\;S_1})\\
 &\leq& C(\|f\|_{L^2(\Om_1)}+\|g\|_{\infty,S_0}+\|p\|_{\infty,S_1}+\|q\|_{\infty,S_1})
 \enn
for some constant $C>0$. The lemma is thus proved.
\end{proof}

In order to prove the unique determination of $S_{1}$, we need the following lemma
which was proved in \cite{KP} (see \cite[Lemma 4.4]{KP}; note that our definition
of the parameter $\eta$ and our assumption that $k_{2}^{2}$ is not a Neumann eigenvalue
of $\Delta w+k_{2}^{2}nw=0$ in $\Om_2$ are different from those in \cite[Lemma 4.4]{KP}).

\begin{lemma}\label{lem3.4}
Assume that $f\in L^{2}(\Om_{2})$ and $g,\eta\in C(S_{1})$ with
$\eta\neq0$ and $\eta\leq0$ on $S_{0}$. Then the following problem has a unique solution
$w\in C^{2}(\Om_{1})\cap C(\ov{\Om_{1}})$ of
 \ben
 \label{wOm2new}\Delta w+k_{2}^{2}nw=f&&{\rm in}\;\;\Om_{2},\\
 \label{wbg}\frac{\pa w}{\pa\nu}+i\eta w=g&&{\rm on}\;\;S_1.
 \enn
Furthermore, there exists a constant $C>0$ such that
\ben
\|w\|_{\infty,\;\ov{\Om}_2}
\leq C(\|f\|_{L^2(\Om_1)}+\|g\|_{\infty,\;S_2}).
\enn
\end{lemma}

We are now in a position to state and prove the main result of this section.

\begin{theorem}\label{thm3}
Assume that $\la_{0}\neq1$, $\la_{1}\neq1$, let $S_{0}$, $\wi{S}_{0}$ be
two penetrable interfaces and let $\Om_{2}$, $\wi{\Om}_{2}$ be two penetrable obstacles
for the corresponding scattering problem. If the far field patterns of the scattered fields
for the same incident plane wave $u^{i}(x)=e^{ik_{0}x\cdot d}$ coincide at a fixed frequency
for all incident direction $d\in S$ and observation direction $\wi{x}\in S$,
then $S_0=\wi{S}_0$, $S_1=\wi{S}_1$.
\end{theorem}

\begin{proof}
We first prove that $S_{0}=\wi{S}_{0}$. Let $G_{0}$ be defined as in Lemma \ref{lem3.1}.
Assume that $S_{0}\neq\wi{S}_{0}$. Then we may assume without loss of generality
that there exists $z_{0}\in S_{0}\ba\ov{\wi{\Om}}$.
Let $B_{2}$ be a small ball centered at $z_{0}$ such that $B_{2}\cap\ov{\wi{\Om}}=\emptyset$.
Choose $h>0$ such that the sequence
 \ben
 z_j:=z_0+\frac{h}{j}\nu(z_{0}), \qquad j=1,2,\ldots,
 \enn
is contained in $G_{0}\cap B_{2}$, where $\nu(z_{0})$ is the outward normal to
$S_{0}$ at $z_{0}$. Using the notations in Lemma \ref{lem3.1} and letting
$(u^{s}_{j},\;v_{j},\;w_{j})$ and $(\wi{u}^{s}_{j},\;\wi{v}_{j},\;\wi{w}_{j})$
be the solutions of (\ref{lHE0})-(\ref{lrc}) with $z=z_j$.
Then, by Lemma \ref{lem3.1}, $u^{s}_{j}=\wi{u}^{s}_{j}:=u_{j}$ in $\ov{G_{0}}$.
Since $z_{0}$ has a positive distance from $\ov{\wi{\Om}}$,
we conclude from the well-posedness of the direct scattering problem that
there exists $C>0$ such that
 \be\label{ujes}
\|u_j\|_{\infty,\;S_0\cap B_2}+\left\|\frac{\pa u_j}{\pa\nu}\right\|_{\infty,\;S_0\cap B_2}
\leq C\qquad\mbox{for all}\;\; j\geq1.
 \en
Choose a small ball $B_1$ with center $z_0$ which is strictly contained in $B_2$. Since
 \ben
\|\Phi_{0}(\cdot,z_j)\|_{0,\;S_0}+\|\Phi_0(\cdot,z_j)\|_{1,\;\al,\;S_0\ba B_1}&\leq& C,\\
\|\frac{\pa\Phi_{0}}{\pa\nu}(\cdot,z_j)\|_{0,\;S_{0}}
   +\|\frac{\pa\Phi_{0}}{\pa\nu}(\cdot,z_j)\|_{0,\;\al,\;S_{0}\ba B_{1}}&\leq& C
 \enn
for some positive constant $C$ independent of $j$, we conclude from Lemma \ref{vapriori} that
 \ben
\|v_j\|_{\infty,\;S_0\ba B_2}+\left\|\frac{\pa v_j}{\pa\nu}\right\|_{\infty,\;S_0\ba B_2}
\leq C\qquad\mbox{for all}\;\; j\geq1.
 \enn
From this it follows that
 \be\label{vj-Phi}
\|\la_{0}v_j-\Phi_0(\cdot,z_j)\|_{\infty,\;S_0\ba B_2}
    &\leq& C\qquad\mbox{for all}\;\;j\geq1,\\ \label{pavj-Phi}
\left\|\la_0\frac{\pa v_j}{\pa\nu}-\frac{\pa\Phi_0(\cdot,z_j)}{\pa\nu}
    \right\|_{\infty,\;S_0\ba B_2}&\leq& C\qquad\mbox{for all}\;\; j\geq1.
 \en
The transmission boundary conditions yield
 \be\label{vj-Phi2}
\|v_j-\Phi_{0}(\cdot,z_j)\|_{\infty,\; S_{0}\cap B_{2}}=\|u_j\|_{\infty,\; S_{0}\cap B_{2}}
   &\leq& C\qquad \mbox{for all}\;\; j\geq1,\\ \label{pavj-Phi2}
\left\|\la_{0}\frac{\pa v_{j}}{\pa\nu}-\frac{\pa \Phi_{0}(\cdot,z_j)}{\pa\nu}
   \right\|_{\infty,\; S_0\cap B_2}=\left\|\frac{\pa u_j}{\pa\nu}\right\|_{\infty,\;S_0\cap B_2}
   &\leq& C\qquad \mbox{for all}\;\; j\geq1.
 \en
Combining (\ref{pavj-Phi}) and (\ref{pavj-Phi2}) yields
 \be\label{pavj-Phi3}
\left\|\la_{0}\frac{\pa v_{j}}{\pa\nu}-\frac{\pa\Phi_0(\cdot,z_j)}{\pa\nu}
\right\|_{\infty,\; S_{0}}\leq C\qquad\mbox{for all}\;\; j\geq1.
 \en
This can be used together with (\ref{vj-Phi}) to prove the estimate
 \be\label{vj-Phi3}
\|\la_{0}v_j-\Phi_{0}(\cdot,z_j)\|_{\infty,S_{0}}\leq C \qquad \mbox{for all}\;\; j\geq1.
 \en
In fact, choose a function $\eta\in C^{2}(S_{0})$ supported in $S_{0}\ba B_{2}$
with $\eta\not\equiv0$ and $\eta\le0$ on $S_0$.
Then $v^{\ast}_{j}:=\la_{0}v_j-\Phi_{0}(\cdot,z_j)$ solves the boundary value problem:
 \ben
 \Delta v^{\ast}_{j}+k_1^2v^{\ast}_{j}
     =(k^2_0-k^2_1)\Phi_0(\cdot,z_j)&&\quad\mbox{in}\;\;\Om_1,\\
 \Delta w^{\ast}_{j}+k_2^2nw^{\ast}_{j}=0&&\quad\mbox{in}\;\;\Om_2,\\
 \frac{\pa v^{\ast}_j}{\pa\nu}+i\eta v^{\ast}_j
    =(\frac{\pa}{\pa\nu}+i\eta)[\la_0v_j-\Phi_0(\cdot,z_j)]&&\quad\mbox{on}\;\;S_0,\\
 v^{\ast}_j-w^{\ast}_{j}=-\Phi_0(\cdot,z_j),\;\;\;
 \frac{\pa v^{\ast}_j}{\pa\nu}-\la_{1}\frac{\pa w^{\ast}_j}{\pa\nu}
    =-\frac{\pa\Phi_{0}(\cdot,z_j)}{\pa\nu}&&\quad\mbox{on}\;\; S_{1},
\enn
where $w^{\ast}_{j}=\la_{0}w_{j}$. Since, by (\ref{vj-Phi}) and (\ref{pavj-Phi3}),
$f:=(k^{2}_{0}-k^{2}_{1})\Phi_{0}(\cdot,z_j)\in L^{2}(\Om_{1}),$
$g:=\left({\pa}/{\pa\nu}+i\eta\right)[\la_{0}v_j-\Phi_{0}(\cdot,z_j)]\in C(S_{0})$,
$p:=-\Phi_{0}(\cdot,z_j)\in C(S_{1})$ and
$q:=-{\pa\Phi_{0}(\cdot,z_j)}/{\pa\nu}\in C(S_{1})$,
then the desired result (\ref{vj-Phi3}) follows from Lemma \ref{lem3.3}.

Now the triangle inequality together with (\ref{vj-Phi2}) and (\ref{vj-Phi3}) implies that
\ben
\|(\la_{0}-1)\Phi_{0}(\cdot,z_j)\|_{\infty, S_{0}\cap B_{2}}
 &\leq& \|\la_{0}\Phi_{0}(\cdot,z_j)-\la_{0}v_j\|_{\infty, S_{0}\cap B_{2}}
   + \|\la_{0}v_j-\Phi_{0}(\cdot,z_j)\|_{\infty, S_{0}\cap B_{2}}\\
 &\leq& C.
\enn
This is a contradiction since $\la_0\neq1$ and
$|\Phi_0(z_0,z_j)|_{\infty,S_0\cap B_2}\rightarrow\infty$ as $j\rightarrow\infty.$
Thus, $S_0=\wi{S}_0$.

Arguing similarly as above, but using the point source $\Phi_{1}$ in $G_{1}$ and
Lemmas \ref{wapriori}, \ref{lem3.2} and \ref{lem3.4}, we can easily prove
that $S_1=\wi{S}_1$. The proof is thus complete.
\end{proof}

\begin{remark}{\rm
Our method can be extended straightforwardly to both the 2D case and the case of a
multilayered medium, and a similar result can be obtained (that is, all the interfaces
between the layered media as well as the shape of the embedded obstacle
can be uniquely determined).
}
\end{remark}

\section{Unique determination of the refractive index $n$}\label{sec4}
\setcounter{equation}{0}

In this section we shall prove a uniqueness theorem for recovering the refractive index $n$
from the far field pattern. In doing this, we need the following completeness result.

\begin{lemma}\label{lem4.1}
The set $\left\{{\pa w(\cdot,d)}/{\pa\nu}|\;d\in S\right\}$ of normal derivatives of
the fields $w(\cdot,d)$ corresponding to incident plane waves with directions $d\in S$
are complete in $L^{2}(S_{1})$ provided $k_{2}^{2}$ is not a Neumann eigenvalue
of $\Delta w+k_{2}^{2}nw=0$ in $\Om_{2}$.
\end{lemma}

\begin{proof}
Let $B_{R}$ be a large ball that contains $\Om$ and such that $k_{0}^{2}$
is not a Dirichlet eigenvalue of $\Delta w+k_{2}^{2}nw=0$ in $B_{R}$.
Then the restriction to $\pa B_{R}$ of the set of plane waves
${e^{ik_{0}x\cdot d}|\;d\in S}$ is complete in $L^{2}(\pa B_{R})$
(see Theorem 5.5 in \cite{CK}). Thus we only have to show that the operator
$\La: L^{2}(\pa B_{R})\rightarrow L^{2}(S_{1})$ has a dense range
$\La\phi={\pa w}/{\pa\nu}$, where $w$ solves the boundary problem (\ref{0HE0})-(\ref{0rc})
and $\phi$ is the boundary data of the interior Dirichlet problem
 \ben
 \label{ui}\Delta u^{i}+k^{2}_{0}u^{i}=0\;\;\;\;{\rm in}\;\;B_{R},
 \qquad u^{i}=\phi\;\;\;\;{\rm on}\;\;\pa B_{R}.
 \enn
A simple repeated application of Green's formulas yields that
the $L^{2}$-adjoint $\La^{\ast}$ of $\La$ is given by
 \ben\label{Laast}
 \La^{\ast}\psi:=\ov{\left\{\frac{\pa u^{\ast}}{\pa\nu}
    -\frac{\pa\wi{u}}{\pa\nu}\right\}}\Big|_{\pa B_R},\;\;\;\psi\in L^2(S_1),
 \enn
where $(u^{\ast},v^{\ast},w^{\ast})$ solves
 \be
\label{aHE0}\Delta u^{\ast}+k_{0}^{2}u^{\ast}=0&&\qquad \mbox{in}\;\;\Om_{0},\\
\label{aHE1}\Delta v^{\ast}+k_{1}^{2}v^{\ast}=0&&\qquad \mbox{in}\;\;\Om_{1},\\
\label{aHE2}\Delta w^{\ast}+k_{2}^{2}nw^{\ast}=0&&\qquad \mbox{in}\;\;\Om_{2},\\
\label{atbc}u^\ast-v^\ast=0,\;\;\frac{\pa u^\ast}{\pa\nu}-\la_0\frac{\pa v^\ast}{\pa\nu}=0
                                                &&\qquad\mbox{on}\;\;S_0,\\
\label{atbc1}v^{\ast}-w^{\ast}=-\ov{\psi},\;\;
                \frac{\pa v^{\ast}}{\pa\nu}-\la_1\frac{\pa w^{\ast}}{\pa\nu}=0
                                                &&\qquad \mbox{on}\;\;S_{1},\\
\label{arc}\lim_{r\rightarrow\infty}r\left(\frac{\pa u^{\ast}}{\pa r}
               -ik_{0}u^{\ast}\right)=0&&\qquad r=|x|,
 \en
and $\wi{u}$ is a solution of the interior Dirichlet problem
 \ben
 \label{wiu}\Delta \wi{u}+k^{2}_{0}\wi{u}=0\;\;\;\;{\rm in}\;\;B_{R},
              \qquad \wi{u}=u^{\ast}\;\;\;\;{\rm on}\;\;\pa B_{R}.
 \enn
We just need to show that $\La^{\ast}$ is injective.
Let $\La^{\ast}\psi=0$. Then we have
${\pa u^{\ast}}/{\pa\nu}={\pa\wi{u}}/{\pa\nu}$ and $u^{\ast}=\wi{u}$ on $\pa B_{R}$.
Define
 \ben
 \wi{v}=\left\{\begin{array}{ll}\ds
        u^{\ast}, & \qquad {\rm in}\;\;\R^{3}\ba\ov{B_{R}},\cr
        \ds\wi{u}, &\qquad {\rm in}\;\;B_{R}
        \end{array}\right.
 \enn
Then $\wi{v}$ is an entire solution to the Helmholtz equation
$\Delta \wi{v}+k^{2}_{0}\wi{v}=0$ in $\R^3$ (see \cite{KBook})
satisfying the radiation condition and therefore must vanish identically in $\R^3$.
Hence, $u^{\ast}=0$ in $\R^{3}\ba\ov{B_{R}}$. By the unique continuation principle,
$u^{\ast}=0$ in $\Om_{0}$. Using the transmission conditions (\ref{atbc})
and Homogren's uniqueness theorem we conclude that $v^{\ast}=0$ in $\Om_{1}$
and therefore $v^{\ast}={\pa v^{\ast}}/{\pa\nu}=0$ on $S_{1}$.
By the transmission conditions (\ref{atbc1}), $\frac{\pa w^{\ast}}{\pa\nu}=0$ on $S_1$.
Thus we have $w^{\ast}=0$ in $\Om_{2}$ since $k_{2}^{2}$ is not a Neumann eigenvalue
of $\Delta w+k_{2}^{2}nw=0$ in $\Om_{2}$. Therefore, $\psi=0$
by the transmission conditions (\ref{atbc1}). This completes the proof.
\end{proof}

Based on this completeness result, we are able to prove the following orthogonality relation.

\begin{lemma}\label{orthogonality}
If the far field patterns for the refractive indices $n$ and $\wi{n}$ coincide,
then for any solution $w\in C^{2}(\ov{\Om_{2}})$ of $\Delta w+k_{2}^{2}nw=0$
in $\Om_{2}$ and any solution $\wi{w}\in C^{2}(\ov{\Om_{2}})$ of
$\Delta\wi{w}+k_{2}^{2}\wi{n}\wi{w}=0$ in $\Om_{2}$ we have the relation:
 \be\label{orth}
 \int_{\Om_{2}}(n-\wi{n})w\wi{w}dx=0.
 \en
\end{lemma}

\begin{proof}
We first prove (\ref{orth}) for the special case $w=w(\cdot,d)$.
By Rellich's lemma and Holmgren's uniqueness theorem,
we have $w(\cdot, d)=\wi{w}(\cdot,d)$ and
${\pa w(\cdot,d)}/{\pa\nu}={\pa\wi{w}(\cdot,d)}/{\pa\nu}$ on $S_{1}$.
From this and the equations satisfied by $w(x,d),\wi{w}(\cdot,d)$ and $\wi{w}$
we deduce that
 \be\label{orth1}
 \int_{\Om_{2}}(n-\wi{n})w(x,d)\wi{w}(x)dx
 &=&-\frac{1}{k_2^2}\int_{\Om_2}[\Delta w(x,d)+k_2^2\wi{n}w(x,d)]\wi{w}(x)dx\cr
 &=&-\frac{1}{k_2^2}\int_{\Om_2}\left\{\Delta[w(x,d)-\wi{w}(x,d)]
    +k_2^2\wi{n}[w(x,d)-\wi{w}(x,d)]\right\}\wi{w}(x)dx\cr
 &=&-\frac{1}{k_2^2}\int_{\Om_2}[w(x,d)-\wi{w}(x,d)](\Delta\wi{w}+k_2^2\wi{n}\wi{w})dx=0,
 \en
where use has been made of Green's first theorem in getting the third equation.

To complete the proof we need to show that any general $w$ can be approximated in $L^2(\Om_2)$
by functions $w(\cdot,d),\;d\in S$. Let us assume the opposite, that is,
the set $\{w(\cdot,d)|\;d\in S\}$ is not dense in $L^2(\Om_2)$ sense in
$W:=\{w\in C^2(\ov{\Om_2})|\;\Delta w+k_{2}^{2}nw=0\;\mbox{in}\;\Om_2\}$.
Then by Hahn-Banach theorem there would be an $f\in L^2(\Om_2)$ such that
 \be\label{fwd}
 \int_{\Om_{2}}f(x)w(x,d)dx=0
 \en
for $w(x,d)$ with all $d\in S$, but for some $w\in W$,
 \be\label{fw}
 \int_{\Om_{2}}f(x)w(x)dx\neq0.
 \en
Let $u\in H^2(\Om_2)$ be a solution to the interior Neumann problem
 \be\label{uf}
 \Delta u+k_2^2nu=f\;\;\;{\rm in}\;\;\Om_2,\qquad\frac{\pa u}{\pa\nu}=0\;\;\;{\rm on}\;\;S_1.
 \en
From (\ref{fwd}) and (\ref{uf}) it follows on using Green's first theorem that
 \ben
 0=\int_{\Om_{2}}f(x)w(x,d)dx=\int_{\Om_{2}}[\Delta u(x)+k_{2}^{2}nu(x)]w(x,d)dx
  =\int_{S_{1}}u\frac{w(x,d)}{\pa\nu}ds.
 \enn
This, together with Lemma \ref{lem4.1},
implies that $u=0$ on $S_1$. Thus, by Green's first theorem we conclude that
 \ben
 \int_{\Om_2}fwdx=\int_{\Om_{2}}(\Delta u+k_{2}^{2}nu)wdx
  =\int_{\Om_2}u(\Delta w+k_{2}^{2}nw)dx=0,
 \enn
which contradicts (\ref{fw}). The proof is thus complete.
\end{proof}

It is well known that the products $w\wi{w}$ of solutions to $\Delta w+k_{2}^{2}nw=0$
and $\Delta \wi{w}+k_{2}^{2}\wi{n}\wi{w}=0$ for two different
refractive indices $n$ and $\wi{n}$ is complete in $L^{2}(\Om_{2})$.
Such a result was first established by Sylvester and Uhlmann \cite{SU87} and
simplified by H\"{a}hner \cite{Ha96}.
The reader is also referred to the argument of Theorem 6.2 in a recent review
paper by Uhlmann \cite{Uhl09}. Using this completeness result and Lemma \ref{orthogonality},
we have in fact proved that $n=\wi{n}$.

\begin{theorem}\label{n-uni}
Assume that the interfaces $S_{j}\;\;(j=0,1)$ are known and $k_2^2$ is not a Neumann eigenvalue of
$\Delta w+k_{2}^{2}nw=0$ in $\Om_2$.
Then the inhomogeneity $n$ is uniquely
determined by the far field pattern $u^{\infty}(\widehat{x},\;d)$ for all $\widehat{x},\;d\in S$.
\end{theorem}

\begin{remark}{\rm
From Lemma \ref{lem4.1} and the argument in the proof of Lemma \ref{orthogonality},
it follows that the Cauchy data of the solution
for two different refractive index $n$ and $\wi{n}$ coincide.
This result can also be extended straightforwardly to
the two dimensional case. Recently, Bukhgeim \cite{Bu08} proved that
the refractive index $n$ can also be uniquely determined from the set of
Cauchy data in the 2D case. Therefore, Theorem \ref{n-uni} also holds for the
two dimensional case.
}
\end{remark}

Combining Theorem \ref{thm3} with Theorem \ref{n-uni}, we have the following uniqueness result.

\begin{theorem}
Assume that $\la_0\;(\neq1)$, $\la_1\;(\neq1)$, $k_0$, $k_1$ and $k_2$ are
given positive numbers and that $k_2^2$ is not a Neumann eigenvalue of
$\Delta w+k_{2}^{2}nw=0$ in $\Om_2$.
Then the interfaces $S_{j}\;\;(j=0,1)$ and the inhomogeneity $n$ are uniquely
determined by the far field pattern $u^{\infty}(\widehat{x},\;d)$ for all $\widehat{x},\;d\in S$.
\end{theorem}

\section*{Acknowledgements}
This work was supported by the NNSF of China under grant No. 10671201.



\begin{thebibliography}{99}

\bibitem{ARS}C. Athanasiadis, A.G. Ramm and I.G. Stratis, Inverse acoustic scattering
by a layered obstacle, {\em In the book: Inverse Problem, Tomography
and Image Processing}, Plenum, New York, 1998, pp. 1-8.







\bibitem{Bu08}A. Bukhgeim,  Recovering the potential from Cauchy data in two dimensions,
{\em J. Inverse Ill-Posed Problems \bf16} (2008), 19-33.

\bibitem{CK83}D. Colton and R. Kress, {\em Integral Equation Methods in Scattering Theory},
Wiley, New York, 1983.

\bibitem{CK}D. Colton and R. Kress, {\em Inverse Acoustic and Electromagnetic
Scattering Theory} (Second Edition), Springer, Berlin, 1998.








\bibitem{Ha96}P. H\"{a}hner, A periodic Faddeev-type solution
operator, {\em J. Diff. Equations \bf128} (1996), 300-308.


\bibitem{Hbook}P. H\"{a}hner, {\em On Acoustic, Electromagnetic
and Elastic Scattering Problems in Inhomogeneous Media},
Habilitationsschrift, G\"{o}ttingen University, 1998.



\bibitem{KBook}A. Kirsch, {\em An Introduction to the Mathematical Theory of Inverse
Problems}, Springer, New York, 1996.

\bibitem{KK93}A. Kirsch and R. Kress, Uniqueness in inverse obstacle scattering,
{\em Inverse Problems \bf9} (1993), 285-299.

\bibitem{KP}A. Kirsch and L. P\"{a}iv\"{a}rinta, On recovering obstacles inside inhomogeneities,
 {\em Math. Meth. Appl. Sci. \bf21} (1998), 619-651.

\bibitem{Kr}R. Kress, Acoustic scattering: Specific theoretical tools. In: Scattering
 (R. Pike, P. Sabatier, eds.), Academic Press, London, 2001, pp. 37-51.







\bibitem{LZ2}X. Liu, B. Zhang, Direct and inverse obstacle scattering problems in a piecewise
homogeneous medium, submitted for publication, 2009 (arXiv:0912.1443v1).


\bibitem{LZH}X. Liu, B. Zhang and G. Hu, Uniqueness in the inverse scattering problem in a
piecewise homogeneous medium, {\em Inverse Problems \bf26}, (2010), 015002.


\bibitem{N88}A. Nachman, Reconstructions from boundary measurements,
{\em Ann. Math. \bf128} (1988), 531-576.


\bibitem{No88}R. Novikov, Multidimensional inverse spectral problems
for the equation $-\Delta\psi+(v(x)+Eu(x))\psi=0$,
{\em Funktsionalny Analizi Ego Prilozheniya \bf22} (1988), 11-12.\\
Transl. {\em Func. Anal. and its Appl. \bf22} (1988), 263-272.



\bibitem{Ramm88}A.G. Ramm, Recovery of the potential from fixed
energy scattering data, {\em Inverse Problems \bf4} (1988), 877-886.

\bibitem{Ramm92}A.G. Ramm, {\em Multidimensional Inverse Scattering
Problems}, Longman \& Wiley, New York, 1992.





\bibitem{SU87}J. Sylvester and G. Uhlmann, A global uniqueness theorem for an inverse
boundary value problem, {\em Ann. Math. \bf125} (1987), 153-169.

\bibitem{Uhl09}G. Uhlmann, Electrical impedance tomography and Calder\'{o}ns problem,
 {\em Inverse Problems \bf25} (2009), 123011.



\bibitem{YG}G. Yan, Inverse scattering by a multilayered obstacle,
{\em Comput. Math. Appl. \bf48} (2004), 1801-1810.



\end{thebibliography}
\end{document}